%% file: main.tex
\documentclass[12pt]{article}
\usepackage[T2A,T1]{fontenc}
\usepackage[utf8]{inputenc}
\usepackage{csquotes}
\usepackage[russian,english]{babel}
\usepackage[utf8]{inputenc} % allow utf-8 input
\usepackage{hyperref}       % hyperlinks
\usepackage{url}            % simple URL typesetting
\usepackage{booktabs}       % professional-quality tables
\usepackage{amsfonts}       % blackboard math symbols
\usepackage{nicefrac}       % compact symbols for 1/2, etc.
\usepackage{microtype}      % microtypography
\def\hmath$#1${\texorpdfstring{{\rmfamily\textit{#1}}}{#1}}
\usepackage{fullpage}
\usepackage{newtxtext}
\usepackage[smallerops]{newtxmath}
\makeatletter
\newenvironment{chapquote}[2][2em]
  {\setlength{\@tempdima}{#1}%
   \def\chapquote@author{#2}%
   \parshape 1 \@tempdima \dimexpr\textwidth-2\@tempdima\relax%
   \itshape}
  {\par\normalfont\hfill--\ \chapquote@author\hspace*{\@tempdima}\par\bigskip}
\makeatother

\title{An Accelerated Lyapunov Function\\ for Polyak's Heavy-Ball on Convex Quadratics}
\author{Antonio Orvieto\\
ETH Z\"urich}
\date{}
\usepackage{biblatex}
\usepackage{xcolor}
\definecolor{emerald}{rgb}{0.41, 0.88, 0.57}
\input{math_commands.tex}

\begin{document}
\maketitle
\begin{abstract}
In 1964, Polyak showed that the Heavy-ball method, the simplest momentum technique, accelerates convergence of strongly-convex problems in the vicinity of the solution. While Nesterov later developed a globally accelerated version, Polyak's original algorithm remains simpler and more widely used in applications such as deep learning. Despite this popularity, the question of whether Heavy-ball is also globally accelerated or not has not been fully answered yet, and no convincing counterexample has been provided.  This is largely due to the difficulty in finding an effective Lyapunov function: indeed, most proofs of Heavy-ball acceleration in the strongly-convex quadratic setting rely on eigenvalue arguments. Our study adopts a different approach: studying momentum through the lens of quadratic invariants of simple harmonic oscillators. By utilizing the modified Hamiltonian of Stormer-Verlet integrators, we are able to construct a Lyapunov function that demonstrates an $O(1/k^2)$ rate for Heavy-ball in the case of convex quadratic problems. This is a promising first step towards potentially proving the acceleration of Polyak's momentum method and we hope it inspires further research in this field.
\end{abstract}

\section{Introduction}
The problem of unconstrained continuous convex optimization consists in finding an element of the set $\argmin_{q\in\R^d} f(q)$, for some lower bounded convex $f:\R^d\to\R$, generally assumed to be regular, e.g., twice continuously differentiable: $f\in C^2(\R^d,\R)$.
\subsection{Acceleration in discrete- and continuous-time}
In 1979 Nemirovsky and Yudin \cite{nemirovskii1979problem} showed that, if $f$ is convex and $L$-smooth\footnote{A differentiable function $f:\R^d\to\R$ is said to be $L$-smooth if it has $L$-Lipschitz gradients.}, no gradient-based optimizer can converge to a solution faster than $O(1/k^2)$, where $k$ is the number of gradient evaluations\footnote{This lower bound holds just for $k<d$ hence it is only interesting in the high-dimensional setting.}. While Gradient Descent~(\texttt{GD}) converges like $O(1/k)$, the optimal rate $O(1/k^2)$ is achieved by the celebrated Accelerated Gradient Descent~(\texttt{AGD}) method, proposed by Nesterov in 1982~\cite{nesterov1983method}: starting from $p_0=0$ and a random $q_0$, the approximation $q_k$ to a problem solution $q^*$ is computed iteratively as\footnote{Many similar writings are possible. Here, we consider the particular version studied in~\cite{su2016} and a physicist notation, where $p_k$ is a velocity variable. This makes the connection to continuous-time cleaner and consistent with recent work on the geometry of momentum methods~\cite{francca2020dissipative}.} 
\begin{equation*}
    \tag{\texttt{AGD}}
    \begin{cases}
    q_{k+1} = q_k + \beta_k h p_k - h^2 \nabla f(q_k + \beta_k h p_k)\\
    p_{k+1} = (q_{k+1}-q_k)/h
    \end{cases}.
\end{equation*}
where $\beta_k = \frac{k-1}{k+2}$ and $h^2$ is the step-size~(we use the notation $h^2$ instead of the standard $\eta$ for a reason which will become apparent in the next sections). Interestingly, the different behaviour of \texttt{GD} and \texttt{AGD} is retained in the continuous-time limit~(as the step-size vanishes), recently studied by Su, Boyd and Candes \cite{su2016}, but already present in the seminal works of Polyak~\cite{polyak1964some} and Gavurin~\cite{gavurin1958nonlinear}:
\begin{equation*}
\tag{\texttt{GD-ODE}}
\dot q + \nabla f(q) = 0;
\end{equation*}
\begin{equation*}
\tag{\texttt{AGD-ODE}}
\ddot q + \frac{r}{t} \dot q + \nabla f(q) = 0.
\end{equation*}
Namely, we have that \texttt{GD-ODE} converges like $O(1/t)$ and \texttt{AGD-ODE}~(with $r\ge3$) like $O(1/t^2)$, where $t>0$ is the time variable. This result gave researchers a new tool to grasp the baffling essence~(see discussion in \cite{allen2014linear,su2016}) of accelerated optimizers, and led to the design of many novel fast interpretable algorithms~\cite{alimisis2019continuous,krichene2015,xu2018accelerated,wilson2019accelerating}.
% \textit{Contribution.} We show how splitting methods composed of a partitioned Runge-Kutta step and a forward Euler step on \texttt{AGD-ODE} lead to acceleration.

\subsection{Evaluating gradients at a shifted position}
\label{sec:preliminary_experiments}
There are two modifications of \texttt{GD} that bring \texttt{AGD} about:
\begin{enumerate}[leftmargin=0.68cm]
    \item  inclusion of the momentum term (i.e. using $\beta_k\ne 0$);
    \item change in gradient extrapolation point:
$\textcolor{blue}{\nabla f(q_k)}\to\textcolor{magenta}{\nabla f(q_k + \beta_k h p_k)}.$
\end{enumerate}
Questions arise immediately:
\begin{center}
    \textit{Are both these modifications necessary for acceleration?\\
    In particular, is evaluating the gradient at non-iterate points\\ crucial or even necessary for acceleration?}
\end{center}
To put these questions in the right historical context, one has to go back to Polyak's 1964 seminal paper~\cite{polyak1964some}, where the very first momentum method was proposed for $C^2$ and $\mu$-strongly-convex problems\footnote{In the strongly-convex case, $\beta_k$ is not monotonically increasing, but is instead chosen to be a constant dependent on the strong-convexity modulus $\mu$, that is $\beta = \left(\frac{\sqrt{L}-\sqrt{\mu}}{\sqrt{L}+\sqrt{\mu}}\right)^2$.}. Using an elegant functional-analytic argument on multistep methods, Polyak proved that momentum alone --- without shifted gradient evaluation (a.k.a. Heavy-ball (\texttt{HB}), see equation below) --- is able to achieve acceleration\footnote{Here to be intended as a dependency of the rate on the square root of the condition number $L/\mu$.} in a neighborhood of the solution. This local argument becomes of course global in the quadratic case (for a simplified proof, see Proposition 1 in~\cite{lessard2016analysis}).
\begin{equation*}
    \tag{\texttt{HB}}
    \begin{cases}
    q_{k+1} = q_k + \beta_k h p_k - h^2 \nabla f(q_k)\\
    p_{k+1} = (q_{k+1}-q_k)/h
    \end{cases}.
\end{equation*}
Despite the many attempts, nobody in the last 56 years has been able to show that \texttt{HB} has a global (i.e. for any initialization) accelerated rate --- neither in the strongly-convex case~(using a fixed momentum) nor in the non-strongly-convex case~(using an increasing $\frac{k-1}{k+2}$ momentum). Beyond the technical difficulty, another plausible reason may also be lack of interest, as the introduction of Nesterov's globally accelerated method in 1982, that overshadowed the conceptually simpler method from Polyak.

However, many researchers in the last decade, supported by numerical evidence and by the success of Heavy-ball in deep learning~\cite{kingma2014adam}, expressed their belief that \texttt{HB} is accelerated:
\vspace{3mm}

\begin{chapquote}{Ghadimi et al. \cite{ghadimi2015global}, 2015}
[\dots] supported by the numerical
simulations we envisage that the convergence factor could be strengthened even further. This
is indeed left as a future work.\vspace{-1mm}
\end{chapquote}

\begin{chapquote}{Gorbunov et al. \cite{gorbunov2019stochastic}, 2019}
Despite the long history of this approach,
there is still an open question whether the heavy ball method converges to the optimum globally with
accelerated rate when the objective function is twice continuous differentiable.\vspace{-3mm}
\end{chapquote}

\begin{chapquote}{Muehlebach and Jordan \cite{muehlebach2020optimization}, 2020}
 Neither the evaluation of the gradient at a
shifted position, nor a specifically engineered damping parameter, as for example proposed
in Nesterov (2004, Sec. 2.2), \textit{seem}\footnote{After talking to the first author, we decided to replace ``\textit{are}'' (as in the original preprint) with ``\textit{seem}'': indeed, the argument in~\cite{muehlebach2020optimization} is asymptotic and therefore somewhat equivalent to the one of Polyak~\cite{polyak1964some}.} necessary.
\vspace{-3mm}
\end{chapquote}
Other researcher believe \texttt{HB} is not accelerated:
\vspace{3mm}

\begin{chapquote}{Shi et al. \cite{shi2018understanding}, 2018}
If we can translate this argument to the discrete case we can understand why \texttt{AGD} achieves
acceleration globally for strongly-convex functions but the Heavy-ball method does not.
\end{chapquote}
While on the theoretical side the opinion is mixed, on the experimental side no numerical simulation\footnote{In \cite{lessard2016analysis}, the authors show that there exist a strongly-convex smooth function such that Heavy-ball does not converge. However, as also pointed out by Ghadimi et al.~\cite{ghadimi2015global}, such function is not $C^2$, and that a big step-size is used --- which violates the convergence conditions of Thm.~4 in~\cite{ghadimi2015global}. As such, this function does not constitute a proper counterexample.} has been able to show that \texttt{HB} is not accelerated. In Figure~\ref{fig:conjecture}, we provide two examples for the non-strongly-convex case (i.e. $\mu$ very small, such that an increasing momentum is preferable, leading $1/k^2$ convergence as opposed to $(1-\sqrt{\mu/L})^k$). In particular, we show that \texttt{HB} is comparable to \texttt{AGD} through the lens of the pathological lower-bounding quadratic example introduced by~\cite{nemirovskii1979problem} and used to construct the $O(1/k^2)$ bound in convex optimization --- at least until the effect of non-trivial strong-convexity becomes dominant (at around $f(q_k)=10^{-6}$).
 \begin{figure}[ht]
\centering
\begin{subfigure}{0.45\textwidth}\centering\includegraphics[width=1\columnwidth]{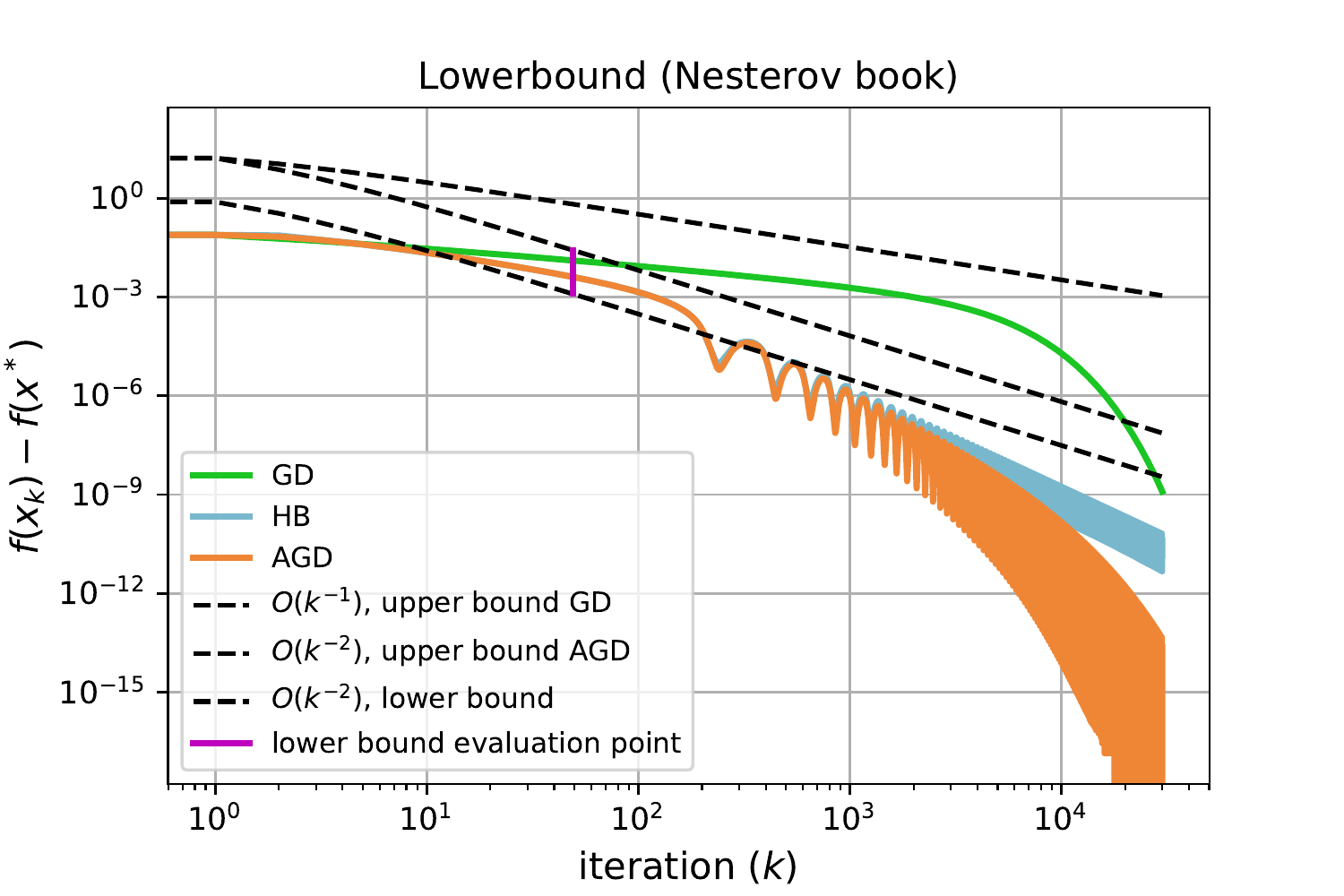}
\caption{\texttt{HB} and \texttt{AGD} on the worst-case (lower bound) quadratic objective from Nesterov~\cite{nesterov2018lectures}.}\end{subfigure}
\hspace{1cm}
\begin{subfigure}{0.45\textwidth}\centering\includegraphics[width=1\columnwidth]{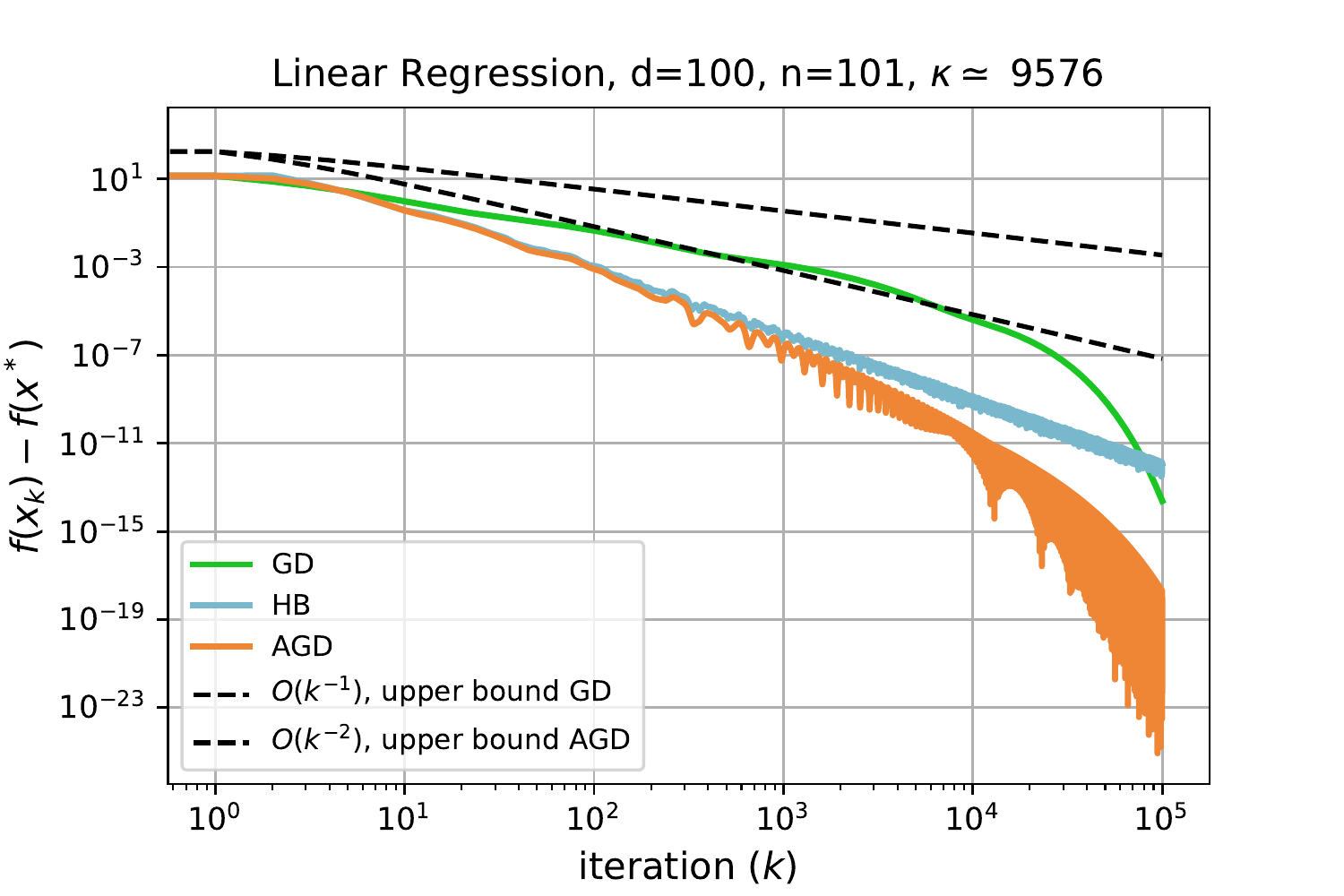}\caption{\texttt{HB} and \texttt{AGD} on ill-conditioned linear regression. The optimal step-size $1/L$ was used.}\end{subfigure}
\caption{For both examples \texttt{HB} with momentum  $\frac{k-1}{k+2}$ exhibits an accelerated $1/k^2$ convergence rate, even though\texttt{ AGD} with momentum  $\frac{k-1}{k+2}$ is faster in a neighborhood of the optimizer due to strong-convexity. Instead, \texttt{GD} violates the Nesterov $O(1/k^2)$ upper bound. We recall that, while Nesterov's upper bound holds for all $k>0$, the $O(1/k^2)$ lower bound~(originally discovered by Nemirovski and Yudin~\cite{nemirovskii1979problem}) only holds at $k = d/2$ (for more details, check the discussion in~\cite{nesterov2018lectures}).}
\label{fig:conjecture}
\vspace{-3mm}
\end{figure}
\subsection{Contributions}
The purpose of the manuscript at hand is to study the effect of shifts in gradient extrapolation points on acceleration in convex optimization (i.e. to study the difference between Heavy-ball and Nesterov's method). In particular, the next pages are organized as follows:
\vspace{-1mm}
\begin{enumerate}[leftmargin=0.48cm]
    \item We start from a continuous-time argument: inspired by a recent idea from Flammarion and Bach~\cite{flammarion2015averaging}, in Section~\ref{sec:quadratic_invariants_cont} we show how \texttt{AGD-ODE} with damping $2/t$ can be derived from the equation of a simple harmonic oscillator: $\ddot u = - Au$. By using Lyapunov equations and a simple change of variables, we retrieve the Lyapunov function proposed by Su, Boyd and Candes~\cite{su2016} to prove a rate $O(1/t^2)$ for \texttt{AGD-ODE}. This procedure is principled and leads to many insights on Lyapunov function design.
    \item In Section~\ref{sec:quadratic_invariants_disc}, we apply the same methodology in discrete time, and show that \texttt{HB} with momentum $\frac{k-1}{k+1}$ can be derived from the St{\"o}rmer--Verlet discretization of the simple harmonic oscillator. Solving again Lyapunov's equations, we are able to show an $O(1/k^2)$ rate for a Heavy-ball argorithm for convex quadratics. While this rate is already present in~\cite{flammarion2015averaging}, our proof technique is different as it relies on a Lyapunov function as opposed to an eigenvalue analysis.
    \item In Section \ref{sec:gener_quadratic}, by generalizing the discrete-time Lyapunov function found in Section~\ref{sec:quadratic_invariants_disc} we derive a modified Heavy-ball method
    $$q_{k+1} = q_k + \frac{k-1}{k+r-1} (q_k-q_{k-1}) - h^2\frac{k+\frac{r-2}{2}}{k+r-1} \nabla f(q_k),$$
    with a rate of convergence $O(1/k^2)$ for any $k\ge2$ and $r\ge2$. Our result not only generalizes the theory in~\cite{flammarion2015averaging}, but also provides an interesting connection between the continuous and the discrete --- as the used Lyapunov function converges, in the limit $h\to 0$, to the one used in~\cite{su2016} for $r\ge2$.
\end{enumerate}

\paragraph{Recent related works.} Very recently, Wang et al.~\cite{wang2022provable} proved that Heavy-ball is accelerated for a class of functions satisfying the Polyak-\L{}ojasiewicz condition. Instead, here we provide a Lyapunov function for the non-strongly-convex setting, where the Polyak-\L{}ojasiewicz constant vanishes. We remark that, for strongly-convex quadratic potentials, Heavy-ball is already known to achieve acceleration~\cite{lessard2016analysis}. However, the eigenvalue argument used in~\cite{lessard2016analysis} cannot be leveraged in the non-strongly-convex setting, where the minimum eigenvalue can be arbitrarily low. As such, our work provides insights on how to construct effective Lyapunov functions in the non-quadratic case, where Lyanonov arguments are often the go-to option.

\section{From quadratic invariants of oscillators to accelerated rates}
\label{sec:quadratic_invariants}

Our procedure in this section is inspired by a beautiful idea presented by Flammarion and Bach~\cite{flammarion2015averaging}: it is sometimes possible to translate a time-dependent convergence rate problem into a time-independent stability problem. Here we go one step further, and show how, with an additional step (computation of quadratic invariants), it is possible to derive Lyapunov functions and rates for the corresponding algorithms. We first illustrate the idea in continuous-time and then proceed with the discrete-time analysis.

Our starting point is the following ODE:
\begin{equation*}
\tag{\texttt{AGD-ODE}2}
\ddot q + \frac{2}{t} \dot q + \nabla f(q) = 0.
\end{equation*}
From the analysis in~\cite{su2016}, we know that on a quadratic $f(q) = f^*+\frac{1}{2}\langle (q-q^*),A(q-q^*)\rangle $, with $A$ positive semidefinite and $f^*\in\R$, the solution converges to $q^*\in\argmin_{x\in\R^d} f(q)$ at the rate $O(1/t^2)$. To prove this rate, the authors in~\cite{su2016} use the following Lyapunov function:
\begin{equation}
    V(q,t) = 2t^2 (f(q)-f^*) + \|t\dot q + (q-q^*)\|^2.
    \label{eq:ly_su_boyd_quadratics}
\end{equation}
 We show here a \textit{constructive way} to derive $V$~(Section~\ref{sec:quadratic_invariants_cont}) and then~(Section~\ref{sec:quadratic_invariants_disc}) we apply the same procedure to get a Lyapunov function for Heavy-ball (i.e., the discretization). For simplicity, we consider here $f^*=0$ and $q^* =0$.
\subsection{Lyapunov functions from continuous-time invariants}
\label{sec:quadratic_invariants_cont}

Consider an harmonic oscillator on the potential $f(u) = \frac{1}{2}u^\top A u$, i.e. $\ddot u = -A u$. From basic physics, we know that such a system is marginally stable (bounded dynamics). By choosing $u = tq$ we get $\dot u = q + t\dot q $ and $\ddot u = \dot q + \dot q + t\ddot q$. This implies
\begin{equation*}
    \dot q + \dot q + t\ddot q = \ddot u = -A t q \ \quad
    \Rightarrow \quad \ \ddot q +\frac{2}{t} \dot q + Aq =0. 
\end{equation*}
    That is, \texttt{AGD-ODE} can be reconstructed from a simple linearized pendulum. By introducing the variable $v = \dot u$, we can write the pendulum in phase space as a linear dynamical system
    $$\begin{pmatrix}
    \dot u\\\dot v
    \end{pmatrix} = \begin{pmatrix}
    0 & I \\ -A& 0
    \end{pmatrix}\begin{pmatrix}
    u\\ v
    \end{pmatrix}.$$
    Hence, the pendulum has the form $\dot y = F y$, where $y = (u,v)$. We would now like to get a Lyapunov function for this system. To do this, we recall a fundamental proposition~(check Thm. 4.6. in \cite{khalil2002nonlinear}).
    \begin{proposition}[Continuous-time Lyapunov equations]
    The linear system $\dot y = Fy$ is Lyapunov stable if and only if  for all positive semidefinite matrices $Q$, there exists a symmetric matrix $P$ such that
    \begin{equation}
        PF+F^TP = -Q.
        \label{eq:lyapunov_eq_cont}
    \end{equation}
    Moreover, $V(y) = y^T P y$ is a Lyapunov function and $\dot V (y) = -y^T Q y$.
    \end{proposition}
    Since we know that a pendulum is only marginally stable (i.e., not asymptotically stable), we can limit ourselves to the choice of a null matrix $Q$. Hence, we need to solve the Lyapunov equation $PF = -F^TP$ for $P$. A solution to this equation~(many exist) is
    $P = \begin{pmatrix}
    A & 0 \\ 0& I
    \end{pmatrix}$, which implies that
    \begin{equation}
        V(u) = \langle u,A u\rangle + \|v\|^2
        \label{eq:energy_continuous}
    \end{equation} is a quadratic invariant, i.e. $\dot V(u)=0$. This is well known, since $V$ is actually twice the total energy (Hamiltonian) of the pendulum. Finally, we can change variables and get that
    \begin{equation}
        V(q) = t^2\langle q,A q\rangle + \left\|\frac{\d}{\d t}(tq)\right\|^2 = 2t^2 f(q) + \|t\dot q + (q-q^*)\|^2,
        \label{eq:ly_quad_cont}
    \end{equation}
    is a Lyapunov function for \texttt{AGD-ODE}2, with $\dot V(q)=0$. This is precisely equation~\ref{eq:ly_su_boyd_quadratics}.

    \paragraph{From quadratic to convex.} With a small modification (using a factor $r-1$ instead of $1$), it is possible to get a Lyapunov function that works for \texttt{AGD-ODE} in the more general convex case.
    
    \begin{proposition}[Theorem 3 from~\cite{su2016}] For convex $L$-smooth objectives, \texttt{AGD-ODE} converges at a rate $O(1/t^2)$. This follows from the fact that 
    \begin{equation}
        V(q,t) = 2t^2 (f(q)-f^*) +\|t\dot q + (r-1)(q-q^*)\|^2.
        \label{eq:lyapunov_su_general}
    \end{equation}
    is a Lyapunov function, for $r\ge3$.
    \end{proposition}
    
\subsection{Discrete-time invariants}
\label{sec:quadratic_invariants_disc}

We apply the construction from the last subsection to the discrete case. Inspired by Flammarion and Bach~\cite{flammarion2015averaging}, we consider at first a slightly modified \texttt{HB}:
\begin{equation*}
\tag{\texttt{HB}$2$}
    q_{k+1} = q_k + \frac{k-1}{k+1} (q_k-q_{k-1}) - h^2 \frac{k}{k+1} \nabla f(q_k).
\end{equation*}

This algorithm is the discrete-time equivalent of $\ddot q + \frac{2}{t} \dot q +\nabla f(q) = 0$. As for the continuous-time case, we start from $f(q) = \frac{1}{2}\langle q,A q \rangle$. In this case, \texttt{HB}$2$ can be written as
$$(k+1) q_{k+1} = 2 k q_k + (k-1) q_{k-1} - h^2 A (kq_k).$$
That is, if we set $u_k = kq_k$, we get
\begin{equation}
    u_{k+1} - 2u_k + u_{k-1} = -h^2 A u_k.
    \label{eq:stormer-verlet}
\end{equation}
With surprise, we recognize that this is the St{\"o}rmer--Verlet method~\cite{hairer2003geometric} on $\ddot u = -A u$, with step-size $h$ (that's why had $h^2$ from the very beginning).

It would be natural, as for the continuous-time case, to consider the total energy as a quadratic invariant to derive a Lyapunov function. However, it turns out that, interestingly, \textit{the St{\"o}rmer--Verlet method does not precisely conserve the total energy}: there are small oscillations~(see Section 3 in~\cite{hairer2014challenges})! Taking into account such small oscillations (Figure~\ref{fig:verlet}) is of fundamental importance --- since they lead to a crucial modification of the invariants we have to use. 

\begin{figure}[ht]
\centering
\includegraphics[width=0.9\textwidth]{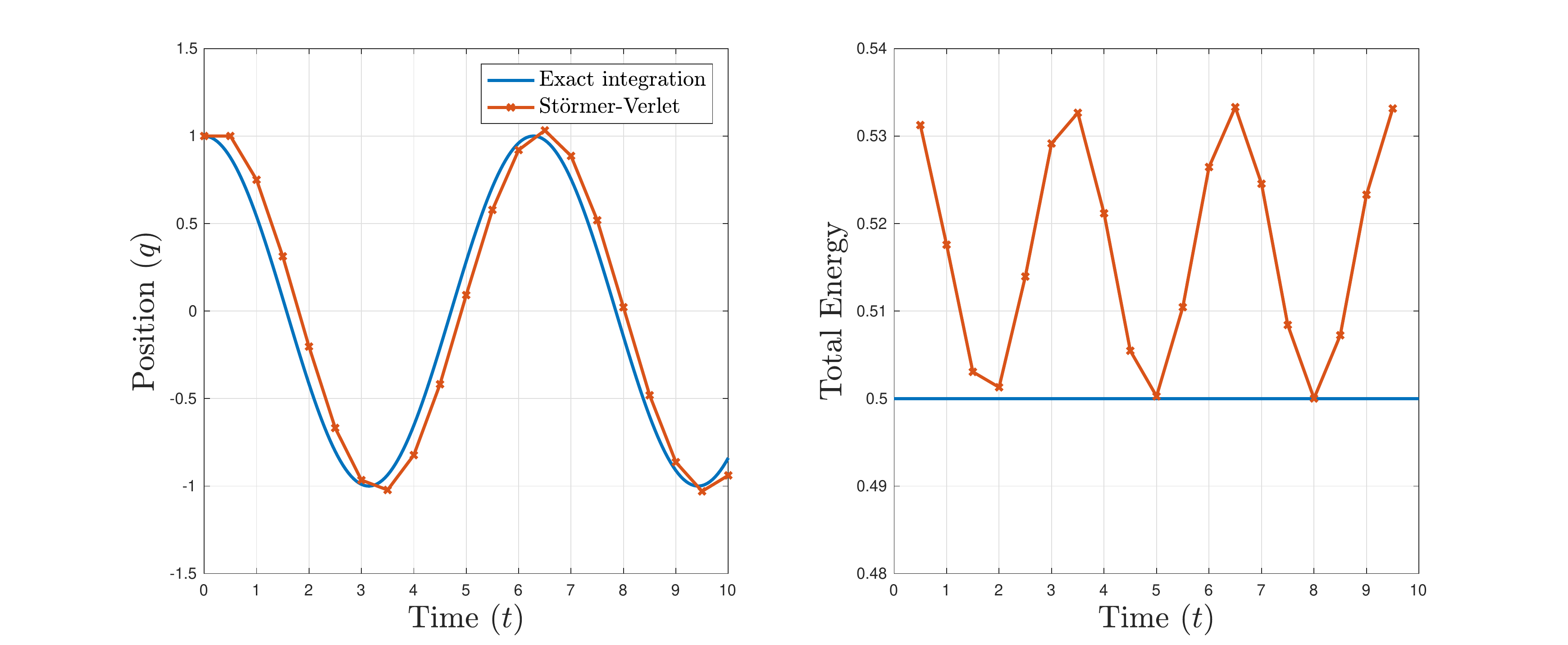}
\caption{The St{\"o}rmer--Verlet method on a one-dimensional quadratic potential (i.e., a simplified pendulum) does not conserve the total energy. Details on this phenomenon can be found in~\cite{hairer2016long,hairer2014challenges}.}
\label{fig:verlet}
\end{figure}

\begin{proposition}[Discrete-time Lyapunov equations]
    The system $ y_{k+1} = Fy_k$ is Lyapunov stable if an only if  for all positive semidefinite matrices $Q$, there exists a symmetric matrix $P$ such that
    \begin{equation}
        F^T P F - P = -Q.
        \label{eq:lyapunov_eq_discrete}
    \end{equation}
    Moreover, $V(y) = y^T P y$ is a Lyapunov function and $V (y_{k+1})-V(y_k) = -y_k^T Q y_k$ for all $k$.
    \end{proposition}
We apply the theorem above~(for $Q=0$) to the linear system
    $$\begin{pmatrix}
    u_{k+1}\\ v_{k+1}
    \end{pmatrix} = \begin{pmatrix}
    I-h^2 A & h I \\ -hA& I
    \end{pmatrix}\begin{pmatrix}
    u_k\\ v_k
    \end{pmatrix}.$$
Under the choice $v_k = (u_{k}-u_{k-1})/h$, this system is equivalent to equation~\ref{eq:stormer-verlet}, i.e., the discretized pendulum we want to find a quadratic invariant for. Solving the discrete Lyapunov equation gives us 
$$P = \begin{pmatrix}
    A & -hA/2 \\ -hA/2& I
    \end{pmatrix},$$
and the associated modified total energy:
\begin{equation}
V(u_k,v_k) = \underbrace{\langle u_k, A u_k\rangle +\|v_k\|^2}_{\text{continuous-time invariant (energy)}}- \underbrace{h\langle v_k, A u_k\rangle}_{\text{vanishing cross-term}}. 
    \label{eq:invariant_stormer}
\end{equation}
We make the following comments:
\begin{itemize}[leftmargin=0.48cm]
    \item As $h\to 0$, the modified energy approaches the total energy in equation~\ref{eq:energy_continuous}. The purpose of the additional cross-term is to eliminate the small energy oscillations we see in Figure~\ref{fig:verlet}.
    \item Assuming without loss of generality that $A$ does not have zero eigenvalues, $P$ is positive semidefinite (i.e., yields a valid Lyapunov function) if and only if the Schur complement of $I$ (i.e., the block $P_{22}$) in $P$ is positive semidefinite. That is, we need 
    \begin{equation}
        B:=A-\frac{h^2}{4}A^2 = A\left(I-\frac{h^2}{4}A\right)\ge0.
        \label{eq:B_matrix}
    \end{equation}
    Since $A$ and $(I-h^2A/4)$ are co-diagonalizable, the product is positive semidefinite if and only if both $A\ge 0$ and $(I-h^2A/4)\ge0$. This requires $0\le A\le \frac{4}{h^2}I$, which in turns implies an upper bound on the step-size $h^2$:
     $$h^2\le\frac{4}{\lambda_{\text{max}}(A)} = \frac{4}{L}.$$
     The same condition~(note that our step-size is $h^2$, not $h$) can be deduced from the analysis of \texttt{HB}$2$ in~\cite{flammarion2015averaging}.
\end{itemize}
Now it's time to change variables back: $u_k = kq_k$. If we set $p_k := (q_k-q_{k-1})/h$, as also done in the introduction, we get
\begin{align*}
    hv_{k} &= h(u_{k}-u_{k-1})/h\\ &= k q_k - (k-1) q_{k-1}+ (k-1) q_{k}- (k-1) q_{k}\\
    &= (k-1)(q_k-q_{k-1})+q_k\\
    &= h(k-1)p_k+q_k.
    \end{align*}
 By substituting these formulas in equation~\ref{eq:invariant_stormer}, we get the following final form for an effective Lyapunov function for \texttt{HB}$2$ --- for the quadratic case:
 \begin{align*}
 &V_k=\langle u_k, A u_k\rangle + \|v_k\|^2-\langle h v_k, A u_k\rangle\\ \Longrightarrow \ & V_k = k^2\langle q_k,Aq_k\rangle + \frac{1}{h^2}\|h(k-1)p_k+q_k\|^2-k\langle h(k-1)p_k+q_k, A q_k\rangle.
 \end{align*}
 To better understand this Lyapunov function, we multiply everything by $h^2$ and get
 $$V_k = (kh)^2\langle q_k,Aq_k\rangle + \|h(k-1)p_k+q_k\|^2 - h^2k\langle h(k-1)p_k+q_k, A q_k\rangle.$$
 Recalling that the ``time'' variable $t$ is defined to be $t_k = hk$, this cost becomes
  $$V_k = t_k^2\langle q_k,Aq_k\rangle + \|t_{k-1}p_k+q_k\|^2 - h t_k\langle t_{k-1}p_k+q_k, A q_k\rangle.$$
This Lyapunov function can be easily generalized by noting that $\langle q_k,Aq_k\rangle = 2(f(q_k)-f^*)$ and $A q_k = \nabla f(q_k)$:
\begin{equation}
    V_k = 2t_k^2 (f(q_k)-f^*) + \|t_{k-1}p_k+q_k\|^2 - t_k\langle\nabla f(q_k), t_{k-1}p_k+q_k\rangle.
    \label{eq:lyapunov_HB_first}
\end{equation}

Finally, note that 
\begin{itemize}[leftmargin=0.48cm]
    \item From equation~\ref{eq:lyapunov_HB_first} as $h\to0$ we get equation~\ref{eq:ly_quad_cont}: the continuous-time Lyapunov function of~\cite{su2016}.
\item  The mixing term is necessary and makes the positive definiteness~(see equation~\ref{eq:B_matrix}) of $V_k$ non-trivial.
\end{itemize} 

All in all, in this subsection, we proved the following result.
\begin{tcolorbox}
\begin{proposition}
    Let $A\in\R^{d\times d}$ be positive semidefinite, $f^*\in\R$ and $f(q)=  f^* + \frac{1}{2}(q-q^*)^T A (q-q^*)$. Let $(q_k)_{k\ge0}$ be the iterates of \texttt{HB}$2$ and $p_k:=(q_k-q_{k-1})/h$. If the step-size $h^2<\frac{4}{\lambda_\text{max}(A)}$, then equation~\ref{eq:lyapunov_HB_first} is non-negative and such that $V_{k+1}-V_k =0$ along the \texttt{HB}$2$ trajectory, for all $k$. From this, one can deduce an accelerated rate of $O(1/k^2)$ in suboptimality.
    \label{prop:proof_quadratic_bach}
\end{proposition}
\end{tcolorbox}
Details of the proof are given in the proof of Theorem~\ref{thm:proof_quadratic}~(more general).

\section{Accelerated Heavy-ball methods for convex quadratics} 
\label{sec:gener_quadratic}
In this section, we start to lift the discussion to the convex non-quadratic setting, by providing a generalization of \texttt{HB}2. Indeed, we know from the continuous-time analysis in~\cite{su2016} that $\ddot q + \frac{2}{t}\dot q + \nabla f(q)=0$ may not have an accelerated rate for functions which are convex but not necessarily quadratic. In this case, a rate of $O(1/t^2)$ only holds~\footnote{The case $0<r\le3$ was studied by Attouch et al.~\cite{attouch2019rate}: a convergence rate of $O(1/t^p)$ with $p<2r/3$ is shown in this case. The same result also holds in discrete time.} for $$\ddot q + \frac{r}{t}\dot q + \nabla f(q)=0,$$ with $r\ge 3$. In the same way, we expect that \texttt{HB}$2$~(which is the discretization for $r=2$) may not have an accelerated rate in the convex non-quadratic setting and a \textit{generalization corresponding to high friction is therefore necessary}.

Our objective in this chapter is to construct such a generalization of \texttt{HB}$2$, which we name \texttt{HB}$r$. 

\subsection{A generalized Heavy-ball with high friction and guarantees on quadratics}
After a few weeks of intense calculations, we found that this algorithm gives the desired result~(Thm.~\ref{thm:proof_quadratic}).
\begin{equation*}
\tag{\texttt{HB}$r$}
    q_{k+1} = \underbrace{q_k + \frac{k-1}{k+r-1} (q_k-q_{k-1})}_{\text{iterate + momentum}} - h^2\underbrace{ \frac{k+\frac{r-2}{2}}{k+r-1} \nabla f(q_k)}_{\text{scaled gradient of iterate}}.
\end{equation*}
First, note that $r=2$ recovers \texttt{HB}$2$ --- which we proved to be accelerated in the last subsection using a novel Lyapunov argument. The second, and perhaps the most crucial, thing to note is that \texttt{HB}$r$ recalls the high friction generalization of \texttt{AGD} proposed by~\cite{su2016}~(see Theorem 6 in their paper):
\begin{equation*}
\tag{\texttt{AGD}$r$}
    q_{k+1} = \underbrace{q_k + \frac{k-1}{k+r-1} (q_k-q_{k-1})}_{\text{iterate + momentum}} - h^2  \underbrace{\nabla f\left(q_k + \frac{k-1}{k+r-1} (q_k-q_{k-1})\right)}_{\text{gradient of [iterate + momentum]}}.
\end{equation*}

Between \texttt{HB}$r$ and \texttt{AGD}$r$ there are a few important differences:
\begin{itemize}[leftmargin=0.48cm]
    \item In \texttt{AGD}$r$ the gradient is evaluated at $q_k + \frac{k-1}{k+r-1} (q_k-q_{k-1})$, while in \texttt{HB}$r$ it is evaluated at $q_k$. 
    \item in \texttt{HB}$r$ the effective step-size~(i.e. what multiplies the gradient) is iteration-dependent, and goes from $h^2/2$ to $h^2$ as $k\to\infty$. We believe this has \textit{not to be regarded as part of the acceleration mechanism}: it is just a small modification needed to make the analysis easier.
    \item Arguably \texttt{HB}$r$~(neglecting the small correction) is conceptually simpler that \texttt{AGD}$r$: compared to \texttt{GD}, only a momentum term is added at each iteration --- and this can be thought of as the source of acceleration.
\end{itemize}

We proceed in proving that \texttt{HB}$r$ is accelerated in the quadratic case.

\begin{tcolorbox}
\begin{theorem}Let $A\in\R^{d\times d}$ be positive semidefinite, $f^*\in\R$ and $f(q)=  f^* + \frac{1}{2}(q-q^*)^T A (q-q^*)$. Let $(q_k)_{k\ge0}$ be the iterates of \texttt{HB}$r$~($r\ge 2$) and $p_k:=(q_k-q_{k-1})/h$. If $h^2\le\frac{4}{\lambda_\text{max}(A)}$, then
    \begin{multline}
    V_k = 2(k+r-2)^2 h^2 (f(q_k)-f^*) + \|h(k-1)p_k+(r-1)(q_k-q^*)\|^2\\ - h^2 (k+r-2)\langle\nabla f(q_k), h(k-1)p_k+(r-1)(q_k-q^*)\rangle
    \label{eq:lyapunov_HB}
    \end{multline}
    is non-negative and such that $V_{k+1}-V_k \le0$ along the \texttt{HB}$r$ trajectory, for all $k$. Moreover, for any $h^2<\frac{4}{\lambda_\text{max}(A)}$, \texttt{HB}$r$ is accelerated. In particular, if $h^2=\frac{2}{\lambda_\text{max}(A)}$, we have the rate
    $$f(q_k)-f^*\le \frac{\lambda_\text{max}(A)V_0}{2(k+r-2)^2}.$$
    \label{thm:proof_quadratic}
\end{theorem}
\end{tcolorbox}

We note a couple of facts about the Lyapunov function $V_k$ in equation~\ref{eq:lyapunov_HB}.
\begin{itemize}[leftmargin=0.48cm]
    \item It reduces to equation~\ref{eq:lyapunov_HB_first} in the case $r=2$. For $r>2$, it is a graspable generalization of equation~\ref{eq:lyapunov_HB_first} --- which we instead derived in a systematic way using Lyapunov equations. The term $(r-1)$ is inspired by the continuous-time limit in equation~\ref{eq:lyapunov_su_general}.
    \item Consider the Lyapunov function above, but without cross term i.e. 
    $$2(k+r-2)^2 h^2 (f(q_k)-f^*) + \|h(k-1)p_k+(r-1)(q_k-q^*)\|^2.$$
    This function works for proving an $O(1/k^2)$ rate for \texttt{AGDr} (it's a Lyapunov function, see Thm. 6 from~\cite{su2016}). Therefore, higher complexity~(i.e., an additional cross term) is needed to study the acceleration of Heavy-ball, when compared to Nesterov's method. 
    \item As $h\to0$, the cross term vanishes $V_k$ converges to equation~\ref{eq:lyapunov_su_general} --- its continuous-time equivalent. Indeed, both \texttt{HBr} and \texttt{AGDr} converge to \texttt{AGD-ODE} as $h\to 0$.
\end{itemize}

\subsection{Proof of the theorem}
It is useful to simplify equation~\ref{eq:lyapunov_HB} and to work with variables $q_k$ and $q_{k-1}$ --- a natural choice in the discrete setting. We split the Lyapunov function into two parts: $V_k = V^1_k + V^2_k$.
\begin{align}
    &V^1_k := 2(k+r-2)^2 h^2 (f(q_k)-f^*)\label{eq:V1}\\ &\quad \quad\quad - h^2 (k+r-2)\langle\nabla f(q_k),(k-1)(q_k-q_{k-1})+(r-1)(q_{k}-q^*)\rangle.\nonumber\\
    & V^2_k := \|(k-1)(q_k-q_{k-1})+(r-1)(q_{k}-q^*)\|^2.\label{eq:V2}
\end{align}

First, we are going to study $V_k^2$ in the non-quadratic case, and then $V_k^1$ in the quadratic case. Theorem~\ref{thm:proof_quadratic} will follow from a combination of the two corresponding lemmata. 

The first lemma shares many similarities with the proof of Theorem 1 in~\cite{ghadimi2015global}.
\begin{lemma}
For any differentiable function $f:\R^d\to\R$ (not necessarily convex or $L$-smooth) and any sequence of iterates $(q_k)_{k\ge0}$ returned by \texttt{HB}$r$, we have:
\begin{align*}
    V_{k+1}^2-V_k^2\ = \ &-h^2(r-1)(2k+r-2)\langle\nabla f(q_k),q_k-q^*\rangle\\
    &-h^2(k-1)(2k+r-2)\langle\nabla f(q_k),q_{k}-q_{k-1}\rangle\\
    &+\frac{h^4}{4}(2k+r-2)^2\|\nabla f(q_k)\|^2,
\end{align*}
where $V^2_k$ is defined in equation~\ref{eq:V2}.
\label{lemma:V2_quad}
\end{lemma}
\begin{proof}
Let $g_{k} := (k-1)(q_k-q_{k-1})+(r-1)(q_{k}-q^*)$, then, 
$$V^2_{k+1}-V^2_k = \|g_{k+1}\|^2-\|g_{k}\|^2 = \langle g_{k+1}+g_{k},g_{k+1}-g_k\rangle.$$
We proceed in computing $g_{k+1}-g_k$. The algorithm symmetric structure here is fundamental:
\begin{align*}
g_{k+1}-g_k = \ & k(q_{k+1}-q_{k})+(r-1)(q_{k+1}-q^*)-(k-1)(q_k-q_{k-1})-(r-1)(q_{k}-q^*)\\ 
= \ & (k+r-1)q_{k+1}-(k+r-1)q_k-(k-1)(q_k-q_{k-1})\\ 
\overset{\text{(\texttt{HB}$r$)}}{=} & - h^2\left(k+\frac{r-2}{2}\right) \nabla f(q_k).
\end{align*}
Instead, $g_{k+1}+g_k$ is slightly more complex.
\begin{align*}
g_{k+1}+g_k = \ & k(q_{k+1}-q_{k})+(r-1)(q_{k+1}-q^*)+(k-1)(q_k-q_{k-1})+(r-1)(q_{k}-q^*)\\ 
= \ & (k+r-1)q_{k+1}+(-k+r-1)q_k+(k-1)(q_k-q_{k-1})-2(r-1)q^*\\ 
\overset{\text{(\texttt{HB}$r$)}}{=} &(k+r-1)q_{k} + (k-1)(q_k-q_{k-1})- h^2\left(k+\frac{r-2}{2}\right) \nabla f(q_k)\\&+(-k+r-1)q_k+(k-1)(q_k-q_{k-1})-2(r-1)q^*\\ 
= \ &2(r-1)(q_{k}-q^*)+2(k-1)(q_k-q_{k-1})- h^2\left(k+\frac{r-2}{2}\right) \nabla f(q_k).
\end{align*}
The proof is concluded by taking the inner product.
\end{proof}

We proceed by computing the difference $V^1_{k+1}-V^1_k$. Our calculations will be very quick, since we can leverage, in the quadratic case, on a simplified expression for $V^1_k$.

\begin{lemma}
 Let $V^1_k$ be defined as equation~\ref{eq:V1}. In the context of Theorem~\ref{thm:proof_quadratic}, we have
$$V^1_k = h^2(k+r-2)(k-1)\langle q_{k-1}-q^*,A(q_k-q^*)\rangle$$
and
\begin{align*}
    V_{k+1}^1-V_k^1\ = & \ h^2(2k+r-2)\langle q_k-q^*, A(q_k-q^*)\rangle\\
    &+h^2(k-1)(2k+r-2)\langle q_k-q^*, A(q_k-q_{k-1})\rangle\\
    &-\frac{h^4}{2} (2k+r-2)k\|A(q_k-q^*)\|^2.
\end{align*}
\vspace{-3mm}
\label{lemma:V1_quad}
\end{lemma}

\begin{proof}
From equation~\ref{eq:V1}, we get
\begin{align*}
    V^1_k = \ & \ (k+r-2)^2 h^2 \langle q_k-q^*, A(q_k-q^*)\rangle\\ &- h^2 (k+r-2)\langle A(q_k-q^*),(k-1)(q_k-q_{k-1})+(r-1)(q_{k}-q^*)\rangle\\
    = \ & \ (k+r-2)^2 h^2 \langle q_k-q^*, A(q_k-q^*)\rangle\\
    &\ - h^2 (k+r-2)\langle A(q_k-q^*),(k+r-2)(q_k-q^*)-(k-1)(q_{k-1}-q^*)\rangle\\
    = \ & \ h^2(k+r-2)(k-1)\langle q_{k-1}-q^*,A(q_k-q^*)\rangle.
\end{align*}

We proceed computing $V^1_{k+1}-V^1_{k}$ using this simplified form:

\begin{align*}
    V^1_{k+1} - V^1_k = \ & \ h^2(k+r-1)k\langle q_{k}-q^*,A(q_{k+1}-q^*)\rangle\\&-h^2(k+r-2)(k-1)\langle q_{k-1}-q^*,A(q_k-q^*)\rangle\\
    = \ & h^2\langle q_{k}-q^*,A \Delta_k\rangle,
\end{align*}
where
\begin{align*}
    \Delta_k := \ & \ (k+r-1)k(q_{k+1}-q^*)-(k+r-2)(k-1)(q_{k-1}-q^*).
\end{align*}
Now, recall the definition of \texttt{HB}$r$:
$$(q_{k+1}-q^*) = (q_k-q^*) + \frac{k-1}{k+r-1} (q_k-q_{k-1}) - h^2 \frac{k+\frac{r-2}{2}}{k+r-1}A(q_k-q*),$$
where we subtracted $q^*$ from both sides. By plugging this into $\Delta_k$, we get
\begin{align*}
    \Delta_k = \ & \ (k+r-1)k(q_k-q^*) + (k-1)k (q_k-q_{k-1}) - \frac{h^2}{2} (2k+r-2)k A(q_k-q*)\\
    &-(k+r-2)(k-1)(q_{k-1}-q^*)\\
    =\ & \ (2k+r-2)k(q_k-q^*)-(2k+r-2)(k-1)(q_{k-1}-q^*)\\&- \frac{h^2}{2} (2k+r-2)k A(q_k-q^*)\\
    =\ & \ (2k+r-2)(q_k-q^*)+(2k+r-2)(k-1)(q_k-q_{k-1})- \frac{h^2}{2} (2k+r-2)k A(q_k-q^*).
\end{align*}
The result follows after taking the inner product $h^2\langle q_{k}-q^*,A \Delta_k\rangle$.
\end{proof}
We are finally ready to prove the result.
\begin{proof}[Proof of Theorem~\ref{thm:proof_quadratic}]
First, we compute $V_{k+1}^1-V_k = (V_{k+1}^1-V_k^1)+(V_{k+1}^2-V_k^2)$ using Lemma~\ref{lemma:V1_quad} and Lemma~\ref{lemma:V2_quad}~(written for quadratic $f$). Next, we show that a certain condition on the step-size implies positivity of $V_k$ and a convergence rate.
\begin{align*}
    (V_{k+1}^1-V_k^1)+(V_{k+1}^2-V_k^2)\ = & \ h^2(2k+r-2)\langle q_k-q^*, A(q_k-q^*)\rangle\\
    &\hcancel{+h^2(k-1)(2k+r-2)\langle q_k-q^*, A(q_k-q_{k-1})\rangle}\\
    &-\frac{h^4}{2} (2k+r-2)k\|A(q_k-q^*)\|^2\\ &-h^2(r-1)(2k+r-2)\langle q_k-q^*, A(q_k-q^*)\rangle\\
    &\hcancel{-h^2(k-1)(2k+r-2)\langle q_k-q^*,A(q_{k}-q_{k-1})\rangle}\\
    &+\frac{h^4}{4}(2k+r-2)^2\|A(q_k-q^*)\|^2.
\end{align*}
Crucially, note that the terms including $\langle q_k-q^*,A(q_{k}-q_{k-1})\rangle$ cancel. This is necessary to make our proof~(or, probably, any proof) work, since such inner product between the gradient and the momentum changes sign (infinitely) many times along the trajectory, and therefore cannot be easily compared to other quantities. For the same reason, in the corresponding continuous-time proof from~\cite{su2016}, the terms including $\langle \nabla f(q), p\rangle$ also perfectly cancel out.

All in all, by collecting some terms, we get
\begin{align*}
    V_{k+1}-V_{k} &= -h^2(r-2)(2k+r-2)\langle q_k-q^*, A(q_k-q^*)\rangle \\& \quad \quad+ \frac{h^4}{4}(2k+r-2)(r-2)\|A(q_k-q^*)\|^2\\
    &=-h^2(r-2)(2k+r-2)\langle q_k-q^*,B(q_k-q^*)\rangle,
\end{align*}
where $B:=A-\frac{h^2}{4}A^2$ is the matrix that we already studied in the context of Lyapunov equations~(see equation~\ref{eq:B_matrix}). Since $r\ge2$, a sufficient condition for $V_{k+1}-V_{k}\le 0$ is $B\ge0$, which holds under $h^2\le\frac{4}{\lambda_\text{max}(A)}$. As a sanity check, the reader can appreciate the fact that, if $r=2$, then $V_{k+1}=V_k$ --- as we already proved in Proposition~\ref{prop:proof_quadratic_bach}~(follows from the fact that $V_k$ solves the Lyapunov equations).

Last, we have to translate the fact that $V_k$ is non-increasing to a convergence rate. This is not trivial in our case, since $V_k$ also contains a cross term which is not necessarily positive. Actually, we do not even know that $V_k\ge0$ yet! Hence, we have to come up with some tricks. We start from rewriting the~(simplified) Lyapunov function:
$$V_k =
h^2(k+r-2)(k-1)\langle q_{k-1}-q^*,A(q_k-q^*)\rangle + \|(k-1)(q_k-q_{k-1})+(r-1)(q_{k}-q^*)\|^2.$$
Now, let us add and subtract a term $c h^2(k+r-2)^2\langle q_{k}-q^*,A(q_k-q^*)\rangle$, with $c>0$. We have:
$$V_k = c h^2(k+r-2)^2\langle q_{k}-q^*,A(q_k-q^*)\rangle + \tilde V_k,$$
with
\begin{align*}
    \tilde V_k :=& -c h^2(k+r-2)^2\langle q_{k}-q^*,A(q_k-q^*)\rangle \\&\quad + h^2(k+r-2)(k-1)\langle q_{k-1}-q^*,A(q_k-q^*)\rangle\\&\quad + \|(k-1)(q_k-q_{k-1})+(r-1)(q_{k}-q^*)\|^2.
\end{align*}
Now, if we show that $\tilde V_k$ is always positive, then $V_{k+1}\le V_k$ for all $k$ implies:
$$c h^2(k+r-2)^2\langle q_{k}-q^*,A(q_k-q^*)\rangle \le c h^2(k+r-2)^2\langle q_{k}-q^*,A(q_k-q^*)\rangle + \tilde V_k= V_k\le V_0,$$
which gives the desired rate:
$$ f(q_k)-f^*=\frac{1}{2}\langle q_{k}-q^*,A(q_k-q^*)\rangle \le \frac{V_0}{2c h^2(k+r-2)^2}.$$
Therefore, we only need to show $\tilde V_k\ge0$. To do this, we introduce two new variables:
\begin{equation*}
    u_{k} := (k+r-2)(q_k-q^*),\quad\quad w_{k} := (k-1)(q_{k-1}-q^*),
\end{equation*}
and get a simplified form for $\tilde V_k$
\begin{align*}
    \tilde V_k &= -c h^2 \langle u_k,Au_k\rangle + h^2\langle v_k,A w_k\rangle + \|u_k-w_k\|^2\\
    &=\langle u_k, (I-ch^2A)u_k\rangle + \|w_k\|^2 - 2\langle u_k,\left(I-\frac{h^2}{2}A\right)w_k\rangle.
\end{align*}
Hence, we just need to show that
$$\tilde P =\begin{pmatrix}
I-ch^2A & -\left(I-\frac{h^2}{2}A\right)\\
-\left(I-\frac{h^2}{2}A\right) & I
\end{pmatrix}$$
is positive definite, for some $c$ and $h^2$. Using the Schur characterization for positive semidefinite matrices, $\tilde P\ge0$ if and only if 
\begin{align*}
    0\le\tilde B(c) &:=I-ch^2A-\left(I-\frac{h^2}{2}A\right)^2\\
    &=I-ch^2A-I-\frac{h^4}{4}A^2+h^2A\\
    &=h^2 A\left(1-c-\frac{h^2}{4}A\right).
\end{align*}
It is clear that $\tilde B(c)$ is positive semidefinite if and only if
$1-c-\frac{h^2}{4}\lambda_{\text{max}}(A)\ge 0$. That is,
$$h^2\le\frac{4(1-c)}{\lambda_\text{max}(A)}.$$
Hence, for any $c\in (0,1)$ we get an acceleration. In particular, in the theorem, we chose $c=\frac{1}{2}$.
\end{proof}

\subsection{Numerical verification of our Lyapunov function} We verify numerically that the Lyapunov function for \texttt{HB}r proposed in equation~\ref{eq:lyapunov_HB} works on quadratics. To more clearly show the effect the inner product correction term, which originated from the quadratic invariant of the St{\"o}rmer--Verlet method, we use here a slightly different notation: $V_k = V^{11}_k+V^{12}_k+V^{2}_k$, with $V^{1}_k = V^{11}_k+V^{12}_k$.
\begin{align*}
    V^{11}_k &:= 2(k+r-2)^2 h^2 (f(q_k)-f^*);\\ 
    V^{12}_k&:= - h^2 (k+r-2)\langle\nabla f(q_k),(k-1)(q_k-q_{k-1})+(r-1)(q_{k}-q^*)\rangle;\\
    V^{2}_k &:= \|(k-1)(q_k-q_{k-1})+(r-1)(q_{k}-q^*)\|^2.
\end{align*}
We recall that, the term $V^{12}_k$ (a.k.a. the \textit{cross-term}) vanishes as $h\to0$, and is indeed not present in the continuous-time limit. We show that this term, which we derived using Lyapunov equations in Sec.~\ref{sec:quadratic_invariants}, plays a fundamental role in ensuring $V_{k+1}-V_k\le0$. In Figure~\ref{fig:lyapunov_HBr_linreg_1} we verify numerically Thm.~\ref{thm:proof_quadratic}. In Figure~\ref{fig:lyapunov_HBr_linreg_3} we show the essential role of $V^{12}$. Here we used $h^2=1/L$, but \texttt{HB}$r$ can take larger steps~(up to $4/L$), while the other algorithms become unstable~(Figure~\ref{fig:HBr_big_step}).

\clearpage

\begin{figure}[ht!]
\centering
\begin{subfigure}{\textwidth}\centering\includegraphics[width=0.8\columnwidth]{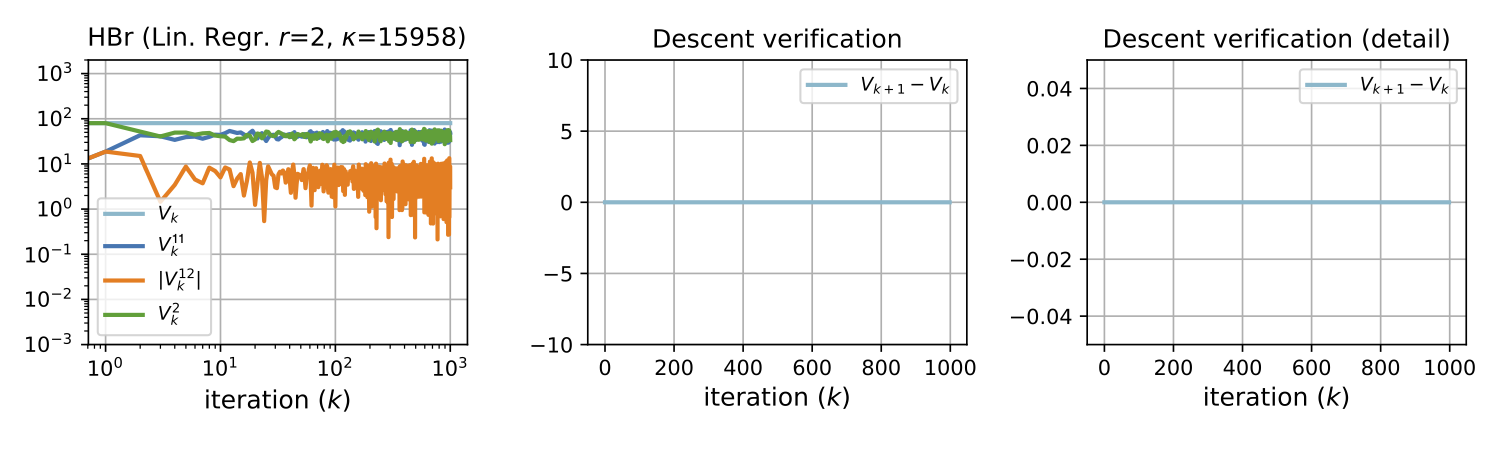}
\end{subfigure}
\vspace{-4mm}
\begin{subfigure}{\textwidth}\centering\includegraphics[width=0.8\columnwidth]{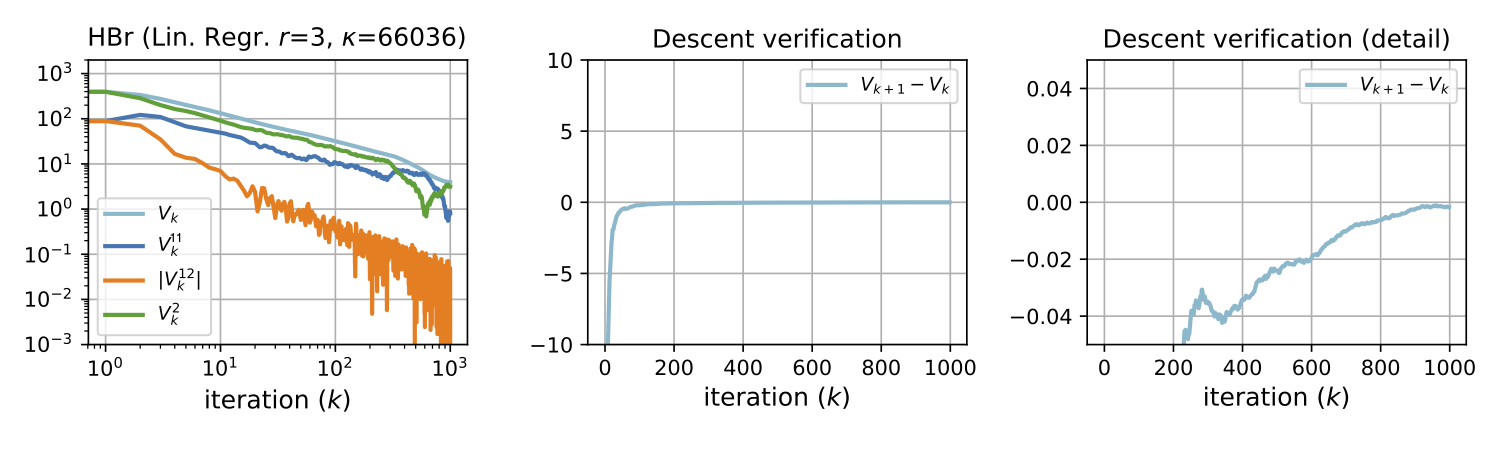}
\end{subfigure}
\vspace{-2mm}
\caption{Dynamics of the Lyapunov function for \texttt{HB}$r$ on linear regression~(ill-conditioned Hessian, with condition number $\kappa$). Shown is the behavior for $r=2,3$ with step-size $1/L$. For $r=2$, $V_{k}$ is constant, as predicted by Prop.~\ref{prop:proof_quadratic_bach}. For $r=3$, $V_{k}$ is decreasing as predicted by Thm.~\ref{thm:proof_quadratic}. }
\label{fig:lyapunov_HBr_linreg_1}
\end{figure}
\vspace{-2mm}
\begin{figure}[ht!]\centering\includegraphics[width=0.8\columnwidth]{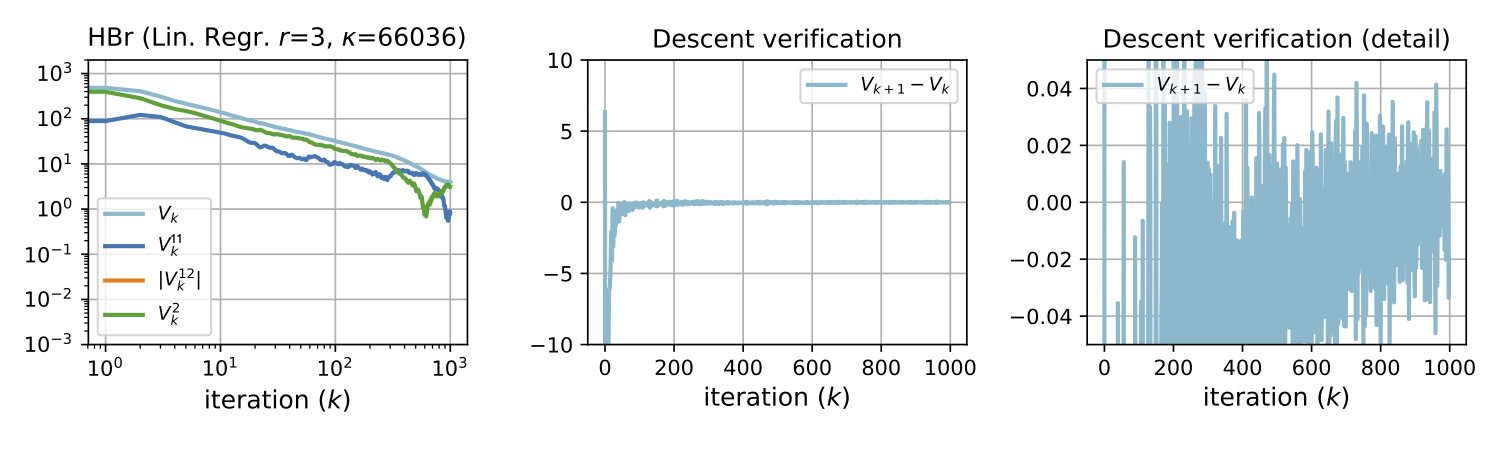}
\vspace{-4mm}
\caption{Same setting of the second example in Figure~\ref{fig:lyapunov_HBr_linreg_1}, but different candidate Lyapunov function (no cross term). This confirms the cross-term is necessary.}
\label{fig:lyapunov_HBr_linreg_3}
\end{figure}
 \vspace{-2mm}
\begin{figure}[ht!]\centering\includegraphics[width=0.35\columnwidth]{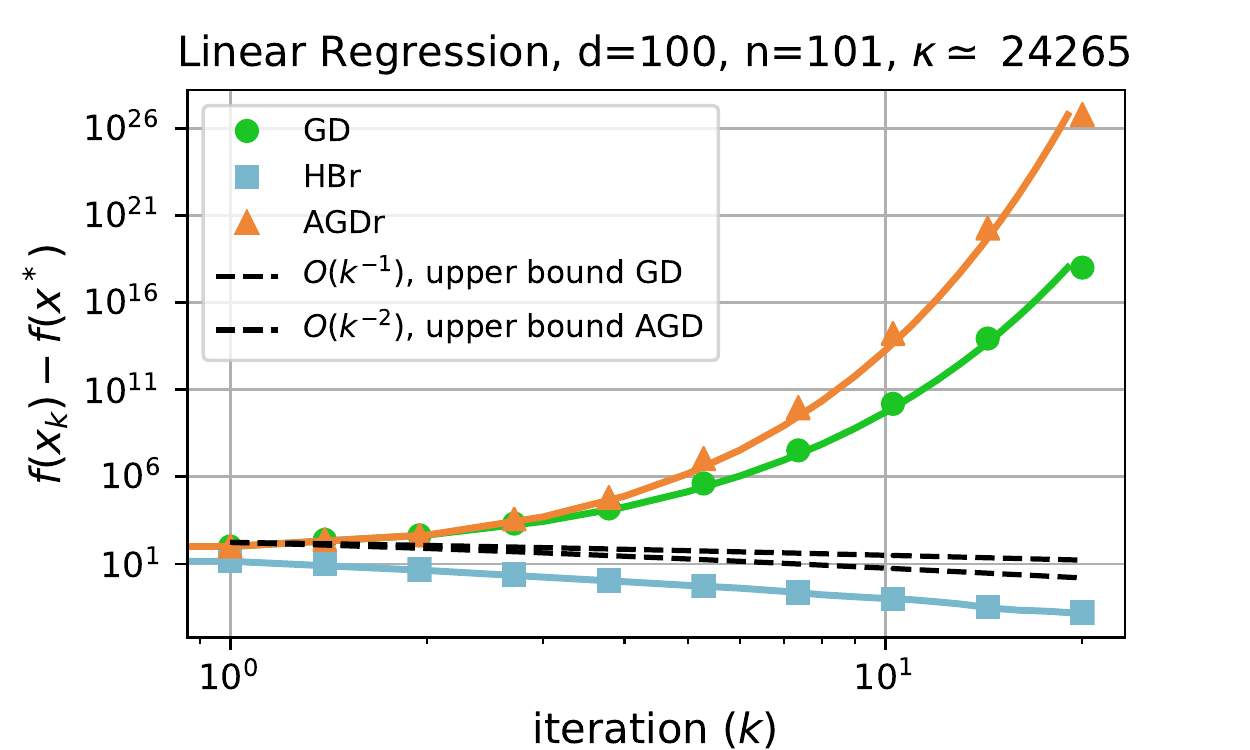}
\vspace{-2mm}
\caption{\texttt{HB}r also works for big step-sizes~(see conditions in Theorem~\ref{thm:proof_quadratic}). Here, used is $h^2=3.9/L$.}
\label{fig:HBr_big_step}
\end{figure} 

\section{Conclusion}
\vspace{-1mm}
In conclusion, the question of whether the Heavy-ball method is globally accelerated for non-strongly-convex quadratic problems has yet to be fully answered, and has attracted the attention of recent research~\cite{wang2022provable}. Our study takes a novel approach by examining momentum through the lens of quadratic invariants of simple harmonic oscillators, and by utilizing the modified Hamiltonian of Stormer-Verlet integrators we were able to construct a Lyapunov function that demonstrates an $O(1/k^2)$ rate for Heavy-ball in the case of convex quadratic problems, where eigenvalues can vanish. This is a promising first step towards potentially proving the acceleration of Polyak's momentum method through Lyapunov function arguments.

\vspace{-1mm}
\section{Acknowledgements}
\vspace{-1mm}
I would like to extend my deepest gratitude to Prof. Boris Polyak, Prof. Christian Lubich, and Konstantin Mishchenko for the stimulating discussions. My appreciation goes to Prof. Aurelien Lucchi and Prof. Thomas Hofmann for their unwavering support and motivation, which helped me to develop the project idea in Spring 2020. Lastly, I cannot express enough my gratitude to Johannes Brahms for his Violinkonzert D-Dur op. 77, which provided the perfect soundtrack to my late-night calculations, igniting my passion and drive to push through the toughest moments.

\printbibliography
\bibstyle{apalike}

\end{document}

%% file: math_commands.tex
\usepackage{amsthm}
\usepackage{calc}
\newsavebox\CBox
\newcommand\hcancel[2][0.5pt]{%
  \ifmmode\sbox\CBox{$#2$}\else\sbox\CBox{#2}\fi%
  \makebox[0pt][l]{\usebox\CBox}%  
  \rule[0.5\ht\CBox-#1/2]{\wd\CBox}{#1}}
\theoremstyle{definition}
\newtheorem{theorem}{Theorem}
\newtheorem{proposition}{Proposition}
\newtheorem{lemma}{Lemma}

\theoremstyle{remark}

\usepackage{mdframed}

\usepackage{tcolorbox}
\tcbset{boxsep=0mm,boxrule=0pt,colframe=white,arc=0mm,left=0.5mm,right=0.5mm}

\usepackage{wrapfig}

\usepackage{amsmath,amsfonts,bm}
\usepackage{subcaption}

\renewcommand{\d}{{\mathrm{d}}}

\def\eqref#1{equation~\ref{#1}}
\def\1{\bm{1}}

\DeclareMathAlphabet{\mathsfit}{\encodingdefault}{\sfdefault}{m}{sl}
\SetMathAlphabet{\mathsfit}{bold}{\encodingdefault}{\sfdefault}{bx}{n}

\newcommand{\R}{\mathbb{R}}

\DeclareMathOperator*{\argmin}{arg\,min}

\usepackage{enumitem}